\documentclass[12pt, onecolumn]{article} 

\usepackage[dvipdfmx]{graphicx}
\usepackage{amssymb,latexsym,amsmath,epsfig,amsthm} 

\usepackage[utf8]{inputenc} 
\usepackage[T1]{fontenc}
\usepackage{url}
\usepackage{ifthen}
\usepackage{cite}
\usepackage{algorithm}
\usepackage{algorithmic}
\usepackage{subfigure}
\usepackage{tikz}
\usepackage{algorithm}
\usepackage{algorithmic}
\usepackage{amsmath}
\usepackage{amssymb}
\usepackage{stmaryrd}
\usepackage{ecltree}
\usepackage{enumitem}
\usepackage{mathrsfs}
\usepackage{xcolor}
\usepackage{comment}
\usepackage{tikz}
\usepackage{bm}
\usepackage{bbm}
\usepackage{booktabs}
\usepackage{wrapfig}
\usepackage{mathtools}
\mathtoolsset{showonlyrefs=true}
\usetikzlibrary{trees}

\usepackage[normalem]{ulem}
\usepackage{here}
\usepackage{cases}
\usepackage{stackengine}
\usepackage{longtable}

\def\BibTeX{{\rm B\kern-.05em{\sc i\kern-.025em b}\kern-.08em
    T\kern-.1667em\lower.7ex\hbox{E}\kern-.125emX}}

\newtheorem{theorem}{Theorem}[section]
\newtheorem{lemma}[theorem]{Lemma}
\newtheorem{proposition}[theorem]{Proposition}
\newtheorem{corollary}[theorem]{Corollary}

\theoremstyle{definition}
\newtheorem{definition}[theorem]{Definition}

\newtheorem{example}[theorem]{Example}

\newcommand{\NN}{{\mathscr{N}}}
\newcommand{\PP}{{\mathscr{P}}}
\newcommand{\es}{{\emptyset}}

\newcommand{\st}{{\ast}}

\newcommand{\osum}{\mathbin{:}}
\newcommand{\add}[3]{{#1}\osum_{{#3}}{#2}}

\newcommand{\uint}{\mathbb{Z}_{\geq 0}}
\newcommand{\pint}{\mathbb{Z}_{\geq 1}}
\newcommand{\short}{\tilde{\mathbb{G}}}

\newcommand{\birth}{\mathrm{b}}
\newcommand{\grundy}{\mathcal{G}}
\newcommand{\vset}{\mathcal{V}}
\DeclareMathOperator*{\mex}{mex}

\newcommand{\eqlab}[2]{\overset{(\mathrm{#1})}{#2}}
\newcommand{\defeq}[1]{\overset{\mathrm{def}}{#1}}
\newcommand{\wh}[1]{\widehat{#1}}

\title{Ordinal Sums with Substitution of \\Impartial Games}
\author{Kengo Hashimoto}
\date{}

\begin{document}
\maketitle

\begin{abstract}
A combinatorial game is a two-player game without hidden information or chance elements.
The disjunctive sum $G + H$ of games $G$ and $H$ is the game in which
$G$ and $H$ are played in parallel, and a player makes a move on exactly one of $G$ and $H$ in a turn.
The ordinal sum $G \osum H$ is similar to the disjunctive sum,
but once the left game $G$ is played, the right game $H$ is discarded and can no longer be played.
It is known that the outcome of a mixture of disjunctive sums and ordinal sums, such as $(G_1 \osum G_2) + ((G_3 + G_4) \osum G_5)$, is determined by the variation sets, the set of Grundy numbers of all options, of the components in the normal-play.
In this paper, we propose a generalization of an ordinal sum, called an ordinal sum with substitution $\add{G}{H}{\wh{H}}$,
which is the game made by combining $G$, $H$, and $\wh{H}$ in the following way:
the games $G$ and $H$ are played in parallel;
a player makes a move on exactly one of $G$ and $H$ in a turn;
each time the left game $G$ is played, the right game $H$ is replaced with $\wh{H}$.
We investigate their fundamental properties and 
prove a simple formula for the variation sets of ordinal sums with substitution.
Apply the formula, we give an explicit expression of the Grundy number of a chain of ordinal sums with substitution consisting of nimbers.
We also provide an example illustrating the generalization of ordinal sums with substitution to poset structures.
\end{abstract}

\section{Introduction}

\subsection{Background}
A \emph{combinatorial game} is a two-player game without hidden information or chance elements. 
In a combinatorial game, two players take turns to make a move alternately.
If it is guaranteed that a winner is determined within a finite number of moves,
then exactly one of the two players has a winning strategy.
The main object of combinatorial game theory is to obtain the \emph{outcome}, which player has a winning strategy, of a given position of a combinatorial game.

The \emph{normal-play} (resp.~\emph{mis\`{e}re-play}) is the way to determine a winner in which a play continues until no further moves can be made, and the player who cannot move on the player's turn loses (resp.~wins). 
The normal-play is actively studied because of its rich structure,
and this paper also assumes the normal-play.

Positions in many well-known rulesets are naturally decomposed into sums of multiple smaller positions.
The \emph{disjunctive sum} $G + H$ of games $G$ and $H$ is the game in which
$G$ and $H$ are played in parallel, and a player makes a move on exactly one of $G$ and $H$ in a turn.
The \emph{ordinal sum} $G \osum H$ is similar to the disjunctive sum,
but once the left game $G$ is played, the right game $H$ is discarded and can no longer be played.

Games in which both players have the same set of possible moves are called \emph{impartial} games. 
The outcome of a disjunctive (resp.~ordinal) sum of impartial games is determined by
the \emph{Grundy numbers} \cite{Gru39, Spr35} (resp.~outcomes) of the components of the sum in the normal-play.
Moreover, the outcome of a mixture of disjunctive sums and ordinal sums is obtained from
the \emph{variation sets}, the set of Grundy numbers of all options, of the components\cite[pp.~219--220]{BCG18}.

\subsection{Contribution and Organization}
In this paper, we focus on impartial games
and propose a generalization of ordinal sums, called \emph{ordinal sums with substitution} $\add{G}{H}{\wh{H}}$,
which is the game made by combining $G$, $H$, and $\wh{H}$ in the following way:
the games $G$ and $H$ are played in parallel;
a player makes a move on exactly one of $G$ and $H$ in a turn;
each time the left game $G$ is played, then the right game $H$ is replaced with $\wh{H}$.

A play by Alice and Bob that begins with $\add{G}{H}{\wh{H}}$ proceeds as follows, for example.
\begin{enumerate}
\item Alice plays the right game $H$ to $H'$. The entire game becomes $\add{G}{H'}{\wh{H}}$.
\item Bob plays the right game $H'$ to $H''$. The entire game becomes $\add{G}{H''}{\wh{H}}$.
\item Alice plays the left game $G$ to $G'$. Since the left game is played, the right game $H''$ is replaced with $\wh{H}$,
so that the entire game becomes $\add{G'}{\wh{H}}{\wh{H}}$.
\item Bob plays the right game $\wh{H}$ to $\wh{H}'$. The entire game becomes $\add{G}{\wh{H}'}{\wh{H}}$.
\item Alice plays the left game $G'$ to $G''$. Since the left game is played, the right game $\wh{H'}$ is replaced with $\wh{H}$,
so that the entire game becomes $\add{G''}{\wh{H}}{\wh{H}}$.
\item Bob plays the left game $G''$ to $G'''$. Since the left game is played, the right game $\wh{H}$ is replaced with $\wh{H}$, but nothing changes, so that the entire game becomes $\add{G'''}{\wh{H}}{\wh{H}}$.
\item They continue playing until the entire game reaches the end position $\add{\es}{\es}{\wh{H}}$.
\end{enumerate}

In Section \ref{subsec:iosum}, we introduce the formal definition of ordinal sums with substitution
and investigate their fundamental properties.

In Section \ref{subsec:iosum-vset}, as one of the main results of this paper,
we prove a simple formula for the variation set $\vset(\cdot)$ of an ordinal sum with substitution as follows, where the notation in the statement is formally introduced in Section \ref{sec:preliminaries}.

\theoremstyle{plain}
\newtheorem*{thm:iosum-vset}{\rm\bf Theorem~\ref{thm:iosum-vset}}
\begin{thm:iosum-vset}
For any $G, H, \wh{H} \in \short$, we have
\begin{align}
\vset(\add{G}{H}{\wh{H}}) = (\vset(G) + \mex\vset(\wh{H})) \osum \vset(H).
\end{align}
\end{thm:iosum-vset}

Using this theorem with the formula (Proposition \ref{prop:vset-sum} (i))
\begin{align}
\vset(G + H) = (\vset(G) \oplus \mex\vset(H)) \cup (\vset(H) \oplus \mex\vset(G))
\end{align}
enables us to obtain the variation set of a mixture of disjunctive sums and ordinal sums with substitution
such as $G_1 \osum_{G_2} (G_3 + (G_4 \osum_{G_5} G_6))$.

We also see the following corollaries to the theorem above on the outcomes $o(\cdot)$ and Grundy numbers $\grundy(\cdot)$ of ordinal sums with substitution.

\newtheorem*{cor:iosum-outcome}{\rm\bf Corollary~\ref{cor:iosum-outcome}}
\begin{cor:iosum-outcome}
For any $G, H, \wh{H} \in \short$, we have
\begin{align}
o(\add{G}{H}{\wh{H}}) =
\begin{cases}
\PP  &\,\,\text{if}\,\, o(H) = \PP \,\,\text{and}\,\, (o(G) = \PP \,\,\text{or}\,\, o(\wh{H}) = \NN),\\
\NN  &\,\,\text{otherwise}.
\end{cases}
\end{align}
\end{cor:iosum-outcome}

\newtheorem*{cor:iosum-grundy}{\rm\bf Corollary~\ref{cor:iosum-grundy}}
\begin{cor:iosum-grundy}~
\begin{enumerate}[label={\rm(\roman*)}]
\item For any $G, H_1, H_2, \wh{H} \in \short$,
if $\grundy(H_1) = \grundy(H_2)$, then $\grundy(\add{G}{H_1}{\wh{H}}) = \grundy(\add{G}{H_2}{\wh{H}})$.
\item For any $G, H, \wh{H}_1, \wh{H}_2 \in \short$,
if $\grundy(\wh{H}_1) = \grundy(\wh{H}_2)$, then $\grundy(\add{G}{H}{\wh{H}_1}) = \grundy(\add{G}{H}{\wh{H}_2})$.
\end{enumerate}
\end{cor:iosum-grundy}

In Section \ref{subsec:iosum-exam}, we provide examples of application of the above results.
We first show that in the case of a chain of nimbers $ \st a_0 \osum_{\st \wh{a}_1} \st a_1 \osum_{\st \wh{a}_2} \st a_2 \osum_{\st \wh{a}_3} \cdots \osum_{\st \wh{a}_n} \st a_n$, the Grundy number is given as follows,
where the operations in the chain are performed from left to right.

\newtheorem*{thm:iosum-star}{\rm\bf Theorem~\ref{thm:iosum-star}}
\begin{thm:iosum-star}
Let $n \in \pint$ and $a_0, a_1, a_2, \ldots, a_n, \wh{a}_1, \wh{a}_2, \ldots, \wh{a}_n \in \uint$ be arbitrary.
Let $p$ be the maximum $p' \in \{1, 2, \ldots, n\}$ such that $\sum_{i = p'}^n (a_i - \wh{a}_i) < 0$, where $p \defeq = 0$ if such $p'$ does not exist.
Then we have
\begin{align}
\grundy \left( \st a_0 \osum_{\st \wh{a}_1} \st a_1 \osum_{\st \wh{a}_2} \st a_2 \osum_{\st \wh{a}_3} \cdots \osum_{\st \wh{a}_n} \st a_n \right) = \sum_{i = p}^n a_i.
\end{align}
\end{thm:iosum-star}

For example, positions of the following ruleset fall into the setting of Theorem \ref{thm:iosum-star}.
\begin{quote}
Initially, there are $n+1$ boxes arranged in a row, referred to as Box $0$, Box $1$, $\ldots$, Box $n$.
For every $i = 0, 1, 2, \ldots, n$, Box $i$ contains $a_i$ stones and is labeled by $\wh{a}_i \in \uint$.
A player's move is to choose a non-empty box and remove one or more stones from the chosen box.
When a player makes a move on Box $i$, an appropriate number of stones are added or removed
to Box $(i+1)$, Box $(i+2)$, $\ldots$, Box $n$ so that the number of stones in each box
becomes $\wh{a}_{i+1}$, $\wh{a}_{i+2}$, $\ldots$, $\wh{a}_n$ written on that box, respectively.
\end{quote}

Also, positions of the following ruleset are reduced to a chain of nimbers in a slightly non-trivial way\footnote{This ruleset is inspired by a problem of competitive programming in \cite{ARC}.}.
\begin{quote}
There are $n$ tokens arranged in a row from left to right, and the $i$-th token from the left has an integer $p_i$ written on it,
where the sequence $(p_1, p_2, \ldots, p_n)$ is a permutation of the sequence $(0, 1, \ldots, n-1)$.
A token is called a \emph{record} if there is no token with a larger integer to its left.
A player's move is selecting one token that is not a record and moving it to the left end of the row.
\end{quote}
We refer to the token with integer $i$ as Token $i$ and define a sequence $(b_0, b_1, \ldots, b_{n-1})$ over $\{0, 1\}$ as
$b_i \defeq= 1$ if and only if Token $i$ is a non-record. 
Note that the operation of moving Token $i$ changes the states of Token $1$, Token $2$, $\ldots$, and Token $(i-1)$ from records to non-records if they are records.
Therefore, a player's move corresponds to choosing an integer $i$ such that $b_i = 1$ and changing the values of the sequence as
$b_0 = b_1 = \cdots = b_{i-1} = 1$ and $b_i = 0$.
In terms of ordinal sums with substitution, this can be interpreted as
the chain of nimbers $\st b_{n-1} \osum_{\st 1} \st b_{n-2} \osum_{\st 1} \cdots \osum_{\st 1} \st b_1 \osum_{\st 1} \st b_0$.
Applying Theorem~\ref{thm:iosum-star}, we can see that the Grundy number of a position is the minimum non-negative integer $i$ such that Token $i$ is a non-record, where it is $n$ if all the tokens are records.

We also consider a generalization of ordinal sums with substitution to a poset structure as follows.
\begin{quote}
A finite poset $(P, \preceq)$ is given.
To each element $x \in P$, an impartial game $G_x$ is assigned independently. 
On a turn, a player chooses an arbitrary $x \in P$ and makes a move on the associated game $G_x$.
At that time, for every $y \in P$ such that $y \preceq x$ and $y \neq x$,
the game $G_y$ is replaced with a game $\wh{G}_y$, which is fixed in advance for each $y \in P$.
\end{quote}
Note that a chain of ordinal sums with substitution $G_0 \osum_{\wh{G}_1} G_1 \osum_{\wh{G}_2} G_2 \osum_{\wh{G}_3} \cdots \osum_{\wh{G}_n} G_n$ corresponds to the special case of the poset being the chain $x_0 \succeq x_1 \succeq \cdots \succeq x_n$.
Depending on the structure of a given poset,
the corresponding game may be represented by recursively taking disjunctive sums and ordinal sums.
We provide an example of the recursive computation of the variation set for a poset structure represented by disjunctive sums and ordinal sums.

\section{Impartial Games}
\label{sec:preliminaries}

Before presenting the main results in Section \ref{sec:main},
we introduce the basic notation and briefly review known results on impartial games in this section.
We refer to the textbooks such as \cite{ANW19, BCG18, Con00, HG16, Sie13} for the detailed discussions
and omit the proofs of the propositions in this section.

Let $\uint$ (resp.~$\pint$) denote the set of all non-negative (resp.~positive) integers.

An impartial game $G$ is identified by the set of all possible transitions by a move from $G$.
More precisely, an element $G$ of the set $\short_n$ of all impartial games that terminate within $n$ moves is identified by
the set of all impartial games reached by a move from $G$, which is a subset of the set $\short_{n-1}$ of all games that terminate within $n-1$ moves.
The set $\short$ of all impartial games is the union of the sets $\short_n$ for all $n \in \uint$.
In this paper, we adopt the following formal definition.

\begin{definition}
\label{def:impartial}
For $n \geq \uint$, the set $\short_n$ is defined as
\begin{align}
\short_n \defeq{=}
\begin{cases}
\{\es\} &\,\,\text{if}\,\,n = 0,\\
2^{\short_{n-1}} &\,\,\text{if}\,\,n \geq 1,\\
\end{cases}
\end{align}
where $2^A$ denotes the power set of a set $A$.
A \emph{game}\footnote{By convention, we refer to an individual game position as a \emph{game} rather than a system of playable rules, such as Chess and Go. Instead, a system of playable rules is referred to as a \emph{ruleset}.}
is an element of
\begin{align}
\short \defeq{=} \bigcup_{n \in \uint} \short_n.
\end{align}
\end{definition}

Note that $\short_0 \subsetneq \short_1 \subsetneq \short_2 \subsetneq \cdots$ by the definition.
For $G \in \short$, the \emph{birthday} $\birth(G)$ of $G$ is defined as the minimum $n \in \uint$ such that $G \in \short_n$.
An element of a game $G$ is called an \emph{option} of $G$, which corresponds to a possible transition from the position by a move.
The empty game $\es$ corresponds to the end position in which neither player can make any more moves.
If $G, H \in \short$ are the same elements of $\short$, then we write $G \cong H$ instead of $G = H$ by convention.

The games in the next definition, called \emph{nimbers}, are fundamental.
\begin{definition}
\label{def:imp-star}
For $n \in \uint$, the game $\st n$ is defined recursively as
\begin{align}
\st n \defeq\cong \{\st n' \mid n' \in \uint, n' < n\}. \label{eq:y2wjh4pe0bmr}
\end{align}
In particular, $\st 0 \defeq\cong \es$.
\end{definition}

\begin{example}
\begin{align}
\short_0 &= \{\es\} = \{\st 0\},\\
\short_1 &= 2^{\short_0} = \{\es, \{\es\}\} = \{\st 0, \st 1\},\\
\short_2 &= 2^{\short_1} = \{\es, \{\es\}, \{\{\es\}\}, \{\es, \{\es\}\}\} = \{\st 0, \st 1, \{\st 1\}, \st 2\}.
\end{align}
\end{example}

\subsection{The Outcome}

For each game, exactly one player has a winning strategy.
Namely, one player can win by repeatedly ``making an appropriate response to the other player's move based on a certain strategy.''
The following mapping indicates which player has a winning strategy for a given game in the normal-play.

\begin{definition}
\label{def:imp-outcome}
We define a mapping $o \colon \short \to \{\PP, \NN\}$ recursively as
\begin{align}
o(G) \defeq=
\begin{cases}
\NN &\,\,\text{if}\,\, o(G') = \PP \,\,\text{for some}\,\, G' \in G,\\
\PP &\,\,\text{if}\,\, o(G') = \NN \,\,\text{for all}\,\, G' \in G.
\end{cases}
\end{align}
For $G \in \short$, $o(G)$ is called the \emph{outcome} of $G$.
\end{definition}

The condition $o(G) = \NN$ indicates that the first player (i.e., the player who makes a move on $G$ next) has a winning strategy in $G$, and $o(G) = \PP$ indicates that the second player does.

\subsection{Disjunctive Sums}

The disjunctive sum $G+H$ of games $G$ and $H$ is the game made by combining $G$ and $H$
in which $G$ and $H$ are played in parallel, and a player makes a move on exactly one of $G$ and $H$ in a turn.
The formal definition is as follows.

\begin{definition}
\label{def:imp-sum}
For $G, H \in \short$,
the \emph{disjunctive sum} $G + H$ of $G$ and $H$ is defined recursively as
\begin{align}
G+H &\defeq{\cong} \{G' + H \mid G' \in G\} \cup \{G + H' \mid H' \in H\}. \label{eq:imp-sum}
\end{align}
\end{definition}

The pair $(\short, +)$ forms a commutative monoid with identity $\es$ as follows.

\begin{proposition}~
\label{prop:imp-monoid}
\begin{enumerate}[label={\rm(\roman*)}]
\item For any $G, H, J \in \short$, we have $(G + H) + J \cong G + (H + J)$.
\item For any $G, H \in \short$, we have $G + H \cong H + G$.
\item For any $G \in \short$, we have $G + \es \cong G$.
\end{enumerate}
\end{proposition}

By Proposition \ref{prop:imp-monoid} (i), we can write $(G+H)+J$ (equivalently, $G+(H+J)$) as $G+H+J$ without ambiguity.

As stated in Proposition \ref{prop:grundy} below,
the outcome of a disjunctive sum $G_1 + G_2 + \cdots + G_n$ of games
is determined by the Grundy number, in the next definition, of each component $G_i$.

\begin{definition}
\label{def:grundy}
For $G \in \short$, the \emph{Grundy number} $\grundy(G)$ of $G$ is defined recursively as
\begin{align}
\grundy(G) \defeq= \mex\{ \grundy(G') \mid G' \in G\},
\end{align}
where $\mex A \defeq= \min (\uint \setminus A)$ for $A \subsetneq \uint$.
\end{definition}

\begin{proposition}[\cite{Gru39, Spr35}]~
\label{prop:grundy}
\begin{enumerate}[label={\rm(\roman*)}]
\item For any $G \in \short$, the following equivalence holds: $o(G) = \PP$ if and only if $\grundy(G) = 0$.
\item For any $G, H \in \short$, we have $\grundy(G+H) = \grundy(G) \oplus \grundy(H)$,
where $\oplus$ denotes the binary addition without carrying, called \emph{nim-sum} or \emph{bitwise XOR}.
\end{enumerate}
\end{proposition}

\subsection{Ordinal Sums}

The ordinal sum $G \osum H$ of games $G$ and $H$ is the game made by combining $G$ and $H$ in the following way:
the games $G$ and $H$ are played in parallel;
a player makes a move on exactly one of $G$ and $H$ in a turn;
once the left game $G$ is played, the right game $H$ is discarded and can no longer be played.
The formal definition is as follows.

\begin{definition}
\label{def:osum}
For $G, H \in \short$,
the \emph{ordinal sum} $G \osum H$ of $G$ and $H$ is defined recursively as
\begin{align}
G \osum H \defeq{\cong} \{G' \mid G' \in G\} \cup \{ (G \osum H') \mid H' \in H\}. \label{eq:lp5v7yq4pmml}
\end{align}
\end{definition}

The pair $(\short, \osum)$ forms a (non-commutative) monoid with identity $\es$ as follows.  
\begin{proposition}~
 \label{prop:osum}
\begin{enumerate}[label={\rm(\roman*)}]
\item For any $G, H, J \in \short$, we have $(G \osum H) \osum J \cong G \osum (H \osum J)$.
\item For any $G \in \short$, we have $(G \osum \emptyset) \cong (\emptyset \osum G) \cong G$.
\end{enumerate}
\end{proposition}

The outcome of an ordinal sum is determined by the outcomes of the components as follows.
\begin{proposition}
\label{prop:osum-outcome}
For any $G, H \in \short$, we have
\begin{align}
o(G \osum H) =
\begin{cases}
\PP  &\,\,\text{if}\,\, o(G) = o(H) = \PP,\\
\NN  &\,\,\text{otherwise}.
\end{cases}
\end{align}
\end{proposition}

\subsection{Disjunctive Sums and Ordinal Sums}
The outcome of a disjunctive sum (resp.~ordinal sum) is determined by the Grundy numbers (resp.~outcomes) of the components by Proposition \ref{prop:grundy} (i) (resp.~Proposition \ref{prop:osum-outcome}).
Then, how should we determine the outcome of a mixture of disjunctive sums and ordinal sums,
such as $(G_1 \osum G_2) + ((G_3 + G_4) \osum G_5)$.
To see the outcome of it, we need the Grundy numbers of ordinal sums.
However, $\grundy(G \osum H)$ is not determined by $\grundy(G)$ and $\grundy(H)$ in general.
Indeed, putting $G \defeq\cong \es, G' \defeq\cong \{\st 1\}$, and $H \defeq\cong \st 2$,
we have $\grundy(G \osum H) = 2 \neq 3 = \grundy(G' \osum H)$ despite of $\grundy(G) = 0 = \grundy(G')$.
Therefore, we need to consider something that can carry more information about a game than the Grundy numbers.
This motivates us to introduce the variation set $\vset(G)$ as follows.

\begin{definition}
\label{def:vset}
For $G \in \short$, the \emph{variation set} $\vset(G)$ of $G$ is defined as
\begin{align}
\vset(G) \defeq= \{\grundy(G') \mid G' \in G\}. \label{eq:6hkn74xb2cdj}
\end{align}
\end{definition}
Namely, the variation set is the set of the Grundy numbers of all options.

The following basic properties are seen directly from Definition \ref{def:vset}.
\begin{proposition}
\label{prop:vset}
For any $G \in \short$, the following statements \rm{(i)}--\rm{(iii)} hold.
\begin{enumerate}[label={\rm(\roman*)}]
\item $\grundy(G) = \mex \vset(G)$.
\item $\vset(G) = \{\mex \vset(G') \mid G' \in G\}$.
\item $o(G) = \PP$ if and only if $\vset(G) \not\owns 0$.
\end{enumerate}
\end{proposition}

By Proposition \ref{prop:vset} (i), 
the Grundy number is recovered from the variation set,
and thus the variation set has more information than the Grundy number.
Also, Proposition \ref{prop:vset} (ii) indicates that
the variation set is obtained from the variation sets of all options.

As seen below in Proposition \ref{prop:vset-sum},
the variation sets of disjunctive sums and ordinal sums are determined by the variation sets of the components.
To state the proposition, we introduce the following notation.

\begin{definition}
\label{def:osum-sets}
For finite sets $A, B \subseteq \uint$,
define $A \osum B \defeq= A \cup \{x_i \mid i \in B\}$,
where $\uint \setminus A = \{x_0, x_1, x_2, \ldots \}$ and $x_0 < x_1 < x_2 < \cdots$.
\end{definition}

Because $\mex (A \osum B) = x_{\mex B}$ by Definition \ref{def:osum-sets},
we see the following observation.

\begin{lemma}
\label{lem:set-osum}
For any finite sets $A, B_1, B_2 \subseteq \uint$,
if $\mex B_1 = \mex B_2$, then $\mex (A \osum B_1) = \mex (A \osum B_2)$.
\end{lemma}

For $A \subseteq \uint$ and $x \in \uint$, we define the set $(A \oplus x)$ as $A \oplus x \defeq= \{a \oplus x \mid a \in A\}$.
Then the following results hold.

\begin{proposition}
\label{prop:vset-sum}
For any $G, H \in \short$, the following statements \rm{(i)} and \rm{(ii)} hold.
\begin{enumerate}[label={\rm(\roman*)}]
\item $\vset(G + H) = (\vset(G) \oplus \mex \vset(H)) \cup (\vset(H) \oplus \mex \vset(G))$.
\item $\vset(G \osum H) = \vset(G) \osum \vset(H)$.
\end{enumerate}
\end{proposition}

Proposition \ref{prop:vset-sum} (i) is seen as
\begin{align}
\vset(G + H)
&\eqlab{A}= \{\grundy(G' + H) \mid G' \in G\} \cup \{\grundy(G + H') \mid H' \in H\}\\
&\eqlab{B}= \{\grundy(G') \oplus \grundy(H) \mid G' \in G\} \cup \{\grundy(G) \oplus \grundy(H') \mid H' \in H\} \\
&= (\{\grundy(G') \mid G' \in G\} \oplus \grundy(H)) \cup (\{\grundy(H') \mid H' \in H\} \oplus \grundy(G))\\
&\eqlab{C}= (\vset(G) \oplus \mex \vset(H)) \cup (\vset(H) \oplus \mex \vset(G)),
\end{align}
where
(A) follows from Equations \eqref{eq:imp-sum} and \eqref{eq:6hkn74xb2cdj},
(B) follows from Proposition \ref{prop:grundy} (ii),
and (C) follows from Equation \eqref{eq:6hkn74xb2cdj} and Proposition \ref{prop:vset} (i).
For Proposition \ref{prop:vset-sum} (ii), refer to \cite[pp.~219--220]{BCG18}.

Lemma \ref{lem:set-osum} and Proposition \ref{prop:vset-sum} (ii) yield the following result, which is a particular case of the Colon principle\cite[p.~219]{BCG18}.
\begin{proposition}
\label{prop:colon-pri}
For any $G, H_1, H_2 \in \short$, if $\grundy(H_1) = \grundy(H_2)$, then $\grundy(G\osum H_1) = \grundy(G \osum H_2)$.
\end{proposition}

\section{Main Results}
\label{sec:main}

In an ordinal sum $G \osum H$, once the left game $G$ is played, the right game $H$ is discarded and can no longer be played.
This can also be interpreted as $H$ being replaced by the empty game $\es$ once $G$ is played.
From this perspective, we propose a generalization of ordinal sums, called \emph{ordinal sums with substitution}, in which $H$ is replaced with a predefined $\wh{H}$ each time $G$ is played.

In Section \ref{subsec:iosum}, we introduce the formal definition of ordinal sums with substitution
and study their basic properties.
In Section \ref{subsec:iosum-vset}, we prove a simple formula for the variation set of an ordinal sum with substitution as
one of the main results of this paper, and see several direct consequences from the theorem.
In Section \ref{subsec:iosum-exam}, we provide examples of applying the above results to several particular cases.

\subsection{Ordinal Sums with Substitution}
\label{subsec:iosum}

The ordinal sum with substitution $\add{G}{H}{\wh{H}}$ is the game made by combining $G$, $H$, and $\wh{H}$ in the following way:
the games $G$ and $H$ are played in parallel;
a player makes a move on exactly one of $G$ and $H$ in a turn;
each time the left game $G$ is played, the right game $H$ is replaced with $\wh{H}$.
The formal definition is as follows.

\begin{definition}
\label{def:iosum}
For $G, H, \wh{H} \in \short$, the game $\add{G}{H}{\wh{H}}$ is defined recursively as
\begin{align}
\add{G}{H}{\wh{H}} \defeq\cong \{ (\add{G'}{\wh{H}}{\wh{H}}) \mid G' \in G\} \cup \{ (\add{G}{H'}{\wh{H}}) \mid H' \in H \}.
\label{eq:iosum}
\end{align}
\end{definition}

The game $\add{G}{H}{\wh{H}}$ is indeed a short game as follows.

\begin{proposition}
\label{prop:iosum-short}
For any $G, H, \wh{H} \in \short$, we have $(\add{G}{H}{\wh{H}}) \in \short_{\beta(G, H, \wh{H})}$,
where $\beta(G, H, \wh{H}) \defeq{=} \birth(G)(\birth(\wh{H})+1) + \birth(H)$.
\end{proposition}

\begin{proof}[Proof of Proposition \ref{prop:iosum-short}]
We prove this by induction on $\beta(G, H, \wh{H})$.
For the base case $\beta(G, H, \wh{H}) = 0$, we have $\birth(G) = \birth(H) = 0$ and thus
\begin{align}
\add{G}{H}{\wh{H}}
\cong \add{\es}{\es}{\wh{H}}
\eqlab{A}\cong \{ (\add{G'}{\wh{H}}{\wh{H}}) \mid G' \in \es\} \cup \{ (\add{\es}{H'}{\wh{H}}) \mid H' \in \es \}
\cong \es
\in \short_0
\end{align}
as desired, where
(A) follows from Equation \eqref{eq:iosum}.
We consider the induction step for $\beta(G, H, \wh{H}) \geq 1$.
For any $G' \in G$, we have
\begin{align}
\beta(G', \wh{H}, \wh{H})
&= \birth(G')(\birth(\wh{H})+1) + \birth(\wh{H})\\
&< \birth(G)(\birth(\wh{H})+1) + \birth(\wh{H})\\
&= \beta(G, H, \wh{H}). \label{eq:44cf5enayvhl}
\end{align}
Hence, we can apply the induction hypothesis to obtain $\add{G'}{\wh{H}}{\wh{H}} \in \short_{\beta(G', \wh{H}, \wh{H})}$,
which leads to
\begin{align}
\{ (\add{G'}{\wh{H}}{\wh{H}}) \mid G' \in G\}
\subseteq \bigcup_{G' \in G} \short_{\beta(G', \wh{H}, \wh{H})}
\subseteq \short_{ \max\{\beta(G', \wh{H}, \wh{H}) \mid G' \in G\}}. \label{eq:zynslyldsn1u}
\end{align}
Also, for any $H' \in H$, we have
\begin{align}
\beta(G, H', \wh{H})
&= \birth(G)(\birth(\wh{H})+1) + \birth(H')\\
&< \birth(G)(\birth(\wh{H})+1) + \birth(H)\\
&= \beta(G, H, \wh{H}). \label{eq:7z7ggm8xf8hq}
\end{align}
Hence, we can apply the induction hypothesis to obtain $\add{G}{H'}{\wh{H}} \in \short_{\beta(G, H', \wh{H})}$,
which leads to
\begin{align}
\{ (\add{G}{H'}{\wh{H}}) \mid H' \in H\}
\subseteq \bigcup_{H' \in H} \short_{\beta(G, H', \widehat{H})}
\subseteq \short_{ \max\{\beta(G, H', \widehat{H}) \mid H' \in H\}}. \label{eq:8pjr9k4yfl88}
\end{align}
Combining the results above, we obtain
\begin{align}
\add{G}{H}{\wh{H}}
&\eqlab{A}\cong \{ (\add{G'}{\wh{H}}{\wh{H}}) \mid G' \in G\} \cup \{ (\add{G}{H'}{\wh{H}}) \mid H' \in H\}\\
&\eqlab{B}\subseteq \short_{ \max\{\beta(G', \wh{H}, \wh{H}) \mid G' \in G\}} \cup \short_{ \max\{\beta(G, H', \widehat{H}) \mid H' \in H\}}\\
&\subseteq \short_{ \max (\{\beta(G', \wh{H}, \wh{H}) \mid G' \in G\} \cup \{\beta(G, H', \widehat{H}) \mid H' \in H\} )},
\end{align}
where
(A) follows from Equation \eqref{eq:iosum},
and (B) follows from Equations \eqref{eq:zynslyldsn1u} and \eqref{eq:8pjr9k4yfl88}.
Therefore, we have
\begin{align}
\add{G}{H}{\wh{H}}
&\in 2^{\short_{ \max (\{\beta(G', \wh{H}, \wh{H}) \mid G' \in G\} \cup \{\beta(G, H', \widehat{H}) \mid H' \in H\} )}}\\
&= \short_{ \max (\{\beta(G', \wh{H}, \wh{H}) \mid G' \in G\} \cup \{\beta(G, H', \widehat{H}) \mid H' \in H\} ) + 1}\\
&\eqlab{A}\subseteq \short_{\beta(G, H, \widehat{H})},
\end{align}
where
(A) follows from Equations \eqref{eq:44cf5enayvhl} and \eqref{eq:7z7ggm8xf8hq}.
\end{proof}

Ordinal sums with replacement have the following fundamental properties.

\begin{proposition}~
\label{prop:iosum}
\begin{enumerate}[label={\rm(\roman*)}]
\item For any $G \in \short$, we have $\add{G}{\es}{\es} \cong G$.
\item For any $H, \wh{H} \in \short$, we have $\add{\es}{H}{\wh{H}} \cong H$.
\item For any $G, H, \wh{H}, J, \wh{J} \in \short$, we have
$\add{(\add{G}{H}{\wh{H}})}{J}{\wh{J}} \cong \add{G}{(\add{H}{J}{\wh{J}})}{  (\add{\wh{H}}{\wh{J}}{\wh{J}})  }$.
\end{enumerate}
\end{proposition}

\begin{proof}[Proof of Proposition \ref{prop:iosum}]
(Proof of (i))
We prove this by induction on $\birth(G)$.
We have
\begin{align}
\add{G}{\es}{\es}
&\eqlab{A}\cong \{ (\add{G'}{\es}{\es}) \mid G' \in G\} \cup \{ (\add{G}{H'}{\es}) \mid H' \in \es \}\\
&\cong \{ (\add{G'}{\es}{\es}) \mid G' \in G\}\\
&\eqlab{B}\cong \{ G' \mid G' \in G\}\\
&\cong G,
\end{align}
where
(A) follows from Equation \eqref{eq:iosum},
and (B) follows from the induction hypothesis.

(Proof of (ii))
We prove this by induction on $\birth(H)$.
We have
\begin{align}
\add{\es}{H}{\wh{H}}
&\eqlab{A}\cong \{ (\add{G'}{\wh{H}}{\wh{H}}) \mid G' \in \es\} \cup \{ (\add{\es}{H'}{\wh{H}}) \mid H' \in H \}\\
&\cong \{ (\add{\es}{H'}{\wh{H}}) \mid H' \in H \}\\
&\eqlab{B}\cong \{ H' \mid H' \in H \}\\
&\cong H,
\end{align}
where
(A) follows from Equation \eqref{eq:iosum},
and (B) follows from the induction hypothesis.

(Proof of (iii))
We prove this by induction on $\birth(G) + \birth(H) + \birth(J)$.
We have
\begin{align}
\lefteqn{\add{(\add{G}{H}{\wh{H}})}{J}{\wh{J}}}\\
&\eqlab{A}\cong \{ (\add{K}{\wh{J}}{\wh{J}}) \mid K \in (\add{G}{H}{\wh{H}})\} \cup \{ (\add{(\add{G}{H}{\wh{H}})}{J'}{\wh{J}}) \mid J' \in J \}\\
&\eqlab{B}\cong \left\{ (\add{K}{\wh{J}}{\wh{J}}) \mid K \in \left(\{ (\add{G'}{\wh{H}}{\wh{H}}) \mid G' \in G\} \cup \{ (\add{G}{H'}{\wh{H}}) \mid H' \in H \} \right)\right\}\\
&\qquad \cup \{ (\add{(\add{G}{H}{\wh{H}})}{J'}{\wh{J}}) \mid J' \in J \}\\
&\cong \{ (\add{(\add{G'}{\wh{H}}{\wh{H}})}{\wh{J}}{\wh{J}}) \mid G' \in G\}\\
&\qquad \cup \{ (\add{(\add{G}{H'}{\wh{H}})}{\wh{J}}{\wh{J}}) \mid H' \in H\}\\
&\qquad \cup \{ (\add{(\add{G}{H}{\wh{H}})}{J'}{\wh{J}}) \mid J' \in J \}, \label{eq:t8mx67yf3awi}
\end{align}
where
(A) follows from Equation \eqref{eq:iosum},
and (B) follows from Equation \eqref{eq:iosum}.
By the symmetric argument, we have
\begin{align}
\add{G}{(\add{H}{J}{\wh{J}})}{  (\add{\wh{H}}{\wh{J}}{\wh{J}})  }
&\cong \{ (\add{G'}{(\add{H}{J}{\wh{J}})}{  (\add{\wh{H}}{\wh{J}}{\wh{J}})  }) \mid G' \in G\}\\
&\qquad \cup \{ (\add{G}{(\add{H'}{J}{\wh{J}})}{  (\add{\wh{H}}{\wh{J}}{\wh{J}})  } \mid H' \in H\}\\
&\qquad \cup \{ (\add{G}{(\add{H}{J'}{\wh{J}})}{  (\add{\wh{H}}{\wh{J}}{\wh{J}})  } \mid J' \in J \}. \label{eq:731d7w6vxzzb}
\end{align}
By the induction hypothesis, the right-hand sides of Equations \eqref{eq:t8mx67yf3awi} and \eqref{eq:731d7w6vxzzb} are identical.
\end{proof}

As seen in Proposition \ref{prop:iosum} (iii),
the result of the operation depends on the placement of parentheses. 
Hereinafter, when parentheses are omitted, it is assumed that the operations are performed from left to right.
Namely, we write
\begin{align}
G_0 \osum_{\wh{G}_1} G_1 \osum_{\wh{G}_2} G_2 \osum_{\wh{G}_3} \cdots \osum_{\wh{G}_n} G_n \label{eq:2fvgec1obliy}
\end{align}
for
\begin{align}
(\ldots((G_0 \osum_{\wh{G}_1} G_1) \osum_{\wh{G}_2} G_2) \osum_{\wh{G}_3} \ldots ) \osum_{\wh{G}_n} G_n. \label{eq:rwujgdysnlmj}
\end{align}

Intuitively, Equation \eqref{eq:2fvgec1obliy} (equivalently, Equation \eqref{eq:rwujgdysnlmj}) is the compound game in which
a player plays one of $G_0, G_1, \ldots, G_n$ and if $G_i$ is played then $G_{i+1}, G_{i+2}, \ldots, G_n$ are replaced with
$\wh{G}_{i+1}, \wh{G}_{i+2}, \ldots, \wh{G}_n$, respectively.

Proposition \ref{prop:iosum} (iii) is naturally generalized as follows.

\begin{corollary}
\label{cor:iosum}
For any $n \in \pint$ and $G_0, G_1, G_2, \ldots, G_n, \wh{G}_1, \wh{G}_2, \ldots, \wh{G}_n \in \short$, we have
\begin{align}
&\lefteqn{G_0 \osum_{\wh{G}_1} G_1 \osum_{\wh{G}_2} G_2 \osum_{\wh{G}_3} \cdots \osum_{\wh{G}_n} G_n}\\
&\cong G_0 \osum_{(\wh{G}_1 \osum_{\wh{G}_2} \wh{G}_2 \osum_{\wh{G}_3} \cdots \osum_{\wh{G}_n} \wh{G}_n)} (G_1 \osum_{\wh{G}_2} G_2 \osum_{\wh{G}_3} \cdots \osum_{\wh{G}_n} G_n).
\end{align}
\end{corollary}

\begin{proof}[Proof of Corollary \ref{cor:iosum}]
We prove this induction on $n$.
The base case $n = 1$ is trivial.
For the induction step $n \geq 2$, we have
\begin{align}
\lefteqn{G_0 \osum_{\wh{G}_1} G_1 \osum_{\wh{G}_2} G_2 \osum_{\wh{G}_3} \cdots \osum_{\wh{G}_{n-1}} G_{n-1} \osum_{\wh{G}_n} G_n}\\
&\eqlab{A}\cong G_0 \osum_{ (\wh{G}_1 \osum_{\wh{G}_2} \wh{G}_2 \osum_{\wh{G}_3} \cdots \osum_{\wh{G}_{n-1}} \wh{G}_{n-1}) } (G_1 \osum_{\wh{G}_2} G_2 \osum_{\wh{G}_3} \cdots \osum_{\wh{G}_{n-1}} G_{n-1}) \osum_{\wh{G}_n} G_n\\
&\eqlab{B}\cong G_0 \osum_{( (\wh{G}_1 \osum_{\wh{G}_2} \wh{G}_2 \osum_{\wh{G}_3} \cdots \osum_{\wh{G}_{n-1}} \wh{G}_{n-1}) \osum_{\wh{G}_{n}} \wh{G}_n )} ((G_1 \osum_{\wh{G}_2} G_2 \osum_{\wh{G}_3} \cdots \osum_{\wh{G}_{n-1}} G_{n-1}) \osum_{\wh{G}_n} G_n)\\
&\cong G_0 \osum_{ (\wh{G}_1 \osum_{\wh{G}_2} \wh{G}_2 \osum_{\wh{G}_3} \cdots \osum_{\wh{G}_{n-1}} \wh{G}_{n-1} \osum_{\wh{G}_{n}} \wh{G}_n) } (G_1 \osum_{\wh{G}_2} G_2 \osum_{\wh{G}_3} \cdots \osum_{\wh{G}_{n-1}} G_{n-1} \osum_{\wh{G}_n} G_n),
\end{align}
where
(A) follows from the induction hypothesis,
and (B) follows from Proposition \ref{prop:iosum} (iii) with
$G \defeq\cong G_0$, $J \defeq\cong G_n$, $\wh{J} \defeq\cong \wh{G}_n$,
\begin{gather}
H \defeq\cong (G_1 \osum_{\wh{G}_2} G_2 \osum_{\wh{G}_3} \cdots \osum_{\wh{G}_{n-1}} G_{n-1}),\\
\wh{H} \defeq\cong (\wh{G}_1 \osum_{\wh{G}_2} \wh{G}_2 \osum_{\wh{G}_3} \cdots \osum_{\wh{G}_{n-1}} \wh{G}_{n-1}).
\end{gather}
\end{proof}

The next proposition provides the relation between ordinal sums with substitution and ordinary ordinal sums.

\begin{proposition}~
\label{prop:iosum-osum}
\begin{enumerate}[label={\rm(\roman*)}]
\item For any $G, H \in \short$, we have $\add{G}{H}{\es} \cong G \osum H$.
\item For any $G, H, \wh{H}, J \in \short$, we have $(\add{G}{H}{\wh{H}}) \osum J \cong \add{G}{ (H \osum J) }{\wh{H}}$.
\item For any $G, H, \wh{H} \in \short$, we have $\add{G}{H}{\wh{H}} \cong (\add{G}{\es}{\wh{H}}) \osum H$.
\end{enumerate}
\end{proposition}

\begin{proof}[Proof of Proposition \ref{prop:iosum-osum}]
(Proof of (i))
We prove this by induction on $\birth(G) + \birth(H)$.
We have
\begin{align}
\add{G}{H}{\es}
&\eqlab{A}\cong \{ (\add{G'}{\es}{\es}) \mid G' \in G\} \cup \{ (\add{G}{H'}{\es}) \mid H' \in H \}\\
&\eqlab{B}\cong \{ G' \mid G' \in G\} \cup \{ (\add{G}{H'}{\es}) \mid H' \in H \}\\
&\eqlab{C}\cong \{ G' \mid G' \in G\} \cup \{ (G \osum H') \mid H' \in H \}\\
&\eqlab{D}\cong G \osum H,
\end{align}
where
(A) follows from Equation \eqref{eq:iosum},
(B) follows from Proposition \ref{prop:iosum} (i),
(C) follows from the induction hypothesis,
and (D) follows from Equation \eqref{eq:lp5v7yq4pmml}.

(Proof of (ii))
We have
\begin{align}
(\add{G}{H}{\wh{H}}) \osum J
&\eqlab{A}\cong \add{(\add{G}{H}{\wh{H}})}{J}{\es} \\
&\eqlab{B}\cong \add{G}{(\add{H}{J}{\es})}{  (\add{\wh{H}}{\es}{\es})  } \\
&\eqlab{C}\cong \add{G}{( H \osum J)}{  (\wh{H} \osum \es)  } \\
&\eqlab{D}\cong \add{G}{ (H \osum J) }{\wh{H}},
\end{align}
where
(A) follows from (i) of this proposition,
(B) follows from Proposition \ref{prop:iosum} (iii),
(C) follows from (i) of this proposition,
and (D) follows from Proposition \ref{prop:osum} (ii)

(Proof of (iii))
We have
\begin{align}
\add{G}{H}{\wh{H}}
\eqlab{A}\cong \add{G}{(\es \osum H)}{\wh{H}}
\eqlab{B}\cong (\add{G}{\es}{\wh{H}}) \osum H,
\end{align}
where
(A) follows from Proposition \ref{prop:osum} (ii),
and (B) follows from (ii) of this proposition.
\end{proof}

\subsection{The Variation Set of Ordinal Sums with Substitution}
\label{subsec:iosum-vset}

As seen in Proposition \ref{prop:vset-sum},
to determine the outcome of a mixture of disjunctive sums and ordinal sums,
it is sufficient to know the variation sets of all components.
So, what happens if we consider a mixture of disjunctive sums and ordinal sums with substitution instead?
At first glance, it might seem that we would need even more detailed information than the variation set for each component.
However, in fact, the variation set alone is sufficient, and the variation set of an ordinal sum with substitution can be given by a simple expression.
We present this as one of the main theorems of this paper as below,
where $A + x \defeq= \{a + x \mid a \in A\}$ for $A \subseteq \uint$ and $x \in \uint$.

\begin{theorem}
\label{thm:iosum-vset}
For any $G, H, \wh{H} \in \short$, we have
\begin{align}
\vset(\add{G}{H}{\wh{H}}) = (\vset(G) + \mex\vset(\wh{H})) \osum \vset(H).
\end{align}
\end{theorem}
Namely, replacing with $\wh{H}$ has the effect of adding $\grundy(\wh{H})$ to the entire variation set.
Using this formula with Proposition \ref{prop:vset-sum} (i) enables us to obtain the variation set of a mixture of disjunctive sums and ordinal sums with substitution.

The proof of Theorem \ref{thm:iosum-vset} relies on the following lemma.
\begin{lemma}
\label{lem:mex-add}
For any finite sets $A, B \subseteq \uint$, we have
\begin{align}
\mex ( (A + b) \osum B ) = a + b,
\end{align}
where $a \defeq= \mex A$ and $b \defeq= \mex B$.
\end{lemma}

\begin{proof}[Proof of Lemma \ref{lem:mex-add}]
Let $\uint \setminus A = \{x_0, x_1, x_2, \ldots \}$ and $x_0 < x_1 < x_2 < \cdots$.\
Then we have $a = x_0$, that is,
\begin{align}
A \owns 0, 1, 2, \ldots, a-1, \,\,\text{and}\,\, A \not\owns a.
\end{align}
Hence, putting $b \defeq{=} \mex B$, we have
\begin{align}
(A+b) &\not\owns 0, 1, 2, \ldots, b-1, \label{eq:jtljke8qgtd1}\\
(A+b) &\owns b, b+1, b+2, \ldots, b+a-1, \label{eq:gmidwozpcfqz}\\
(A+b) &\not\owns b+a. \label{eq:4gtqcw0kkm81}
\end{align}
Let $\uint \setminus (A+b) = \{y_0, y_1, y_2, \ldots \}$ and $y_0 < y_1 < y_2 < \cdots$.
Then by Equations \eqref{eq:jtljke8qgtd1}, \eqref{eq:gmidwozpcfqz}, and \eqref{eq:4gtqcw0kkm81}, we have
\begin{align}
y_i &= i \qquad \text{for all}\,\, i \in \{0, 1, 2, \ldots, b-1\}, \label{eq:aolqdyjpkzkb}\\
y_b &= b+a. \label{eq:lb7gti4oqwnr}
\end{align}
Since $b = \mex B$, we have
\begin{align}
B &\owns 0, 1, 2, \ldots, b-1 \label{eq:knm0un03axrr}
\end{align}
and $B \not\owns b$,
so that
\begin{align}
\{y_i \mid i \in B\} \not\owns y_b \eqlab{A}= b+a, \label{eq:sqmfjkc664ry}
\end{align}
where (A) follows from Equation \eqref{eq:lb7gti4oqwnr}.
Therefore, we have
\begin{align}
(A+b)\osum B
&= (A+b) \cup \{y_i \mid i \in B\}\\
&\eqlab{A}\supseteq \{b, b+1, b+2, \ldots, b+a-1\} \cup \{y_0, y_1, y_2, \ldots, y_{b-1}\}\\
&\eqlab{B}= \{b, b+1, b+2, \ldots, b+a-1\} \cup \{0, 1, 2, \ldots, b-1\}\\
&= \{0, 1, 2, \ldots, b+a-1\} \label{eq:mbgi488fuy8c}
\intertext{and}
(A+b)\osum B &= (A+b) \cup \{y_i \mid i \in B\} \eqlab{C}{\not\owns} b+a, \label{eq:y81swp7jhunp}
\end{align}
where
(A) follows from Equations \eqref{eq:gmidwozpcfqz} and \eqref{eq:knm0un03axrr},
(B) follows from Equation \eqref{eq:aolqdyjpkzkb},
and (C) follows from Equations \eqref{eq:4gtqcw0kkm81} and \eqref{eq:sqmfjkc664ry}.
Equations \eqref{eq:mbgi488fuy8c} and \eqref{eq:y81swp7jhunp} lead to $\mex ((A+b)\osum B) = a + b$ as desired.
\end{proof}

Using the lemma above, we prove the theorem as follows.
\begin{proof}[Proof of Theorem \ref{thm:iosum-vset}]
We prove this by induction on
\begin{align}
\beta(G, H, \widehat{H}) \defeq= \birth(G)(\birth(\wh{H})+1) + \birth(H).
\end{align}
Define $\wh{h} \defeq= \mex\vset(\wh{H})$.
We consider the following two cases.

\vskip 5pt\noindent {\tt Case 1:} $H \not\cong \es$.
We have
\begin{align}
\vset(\add{G}{H}{\wh{H}})
&\eqlab{A}= \vset( (\add{G}{\es}{\wh{H}}) \osum H)\\
&\eqlab{B}= \vset(\add{G}{\es}{\wh{H}}) \osum \vset(H)\\
&\eqlab{C}= ( (\vset(G) + \wh{h}) \osum \vset(\es) ) \osum \vset(H)\\
&= ( (\vset(G) + \wh{h}) \osum \es ) \osum \vset(H)\\
&= (\vset(G) + \wh{h}) \osum \vset(H)
\end{align}
as desired, where
(A) follows from Proposition \ref{prop:iosum-osum} (iii),
(B) follows from Proposition \ref{prop:vset-sum} (ii),
and (C) follows from the induction hypothesis because
\begin{align}
\beta(G, \es, \wh{H})
= \birth(G)(\birth(\wh{H}) + 1)
< \birth(G)(\birth(\wh{H}) + 1) + \birth(H)
= \beta(G, H, \wh{H})
\end{align}
by $H \not\cong \es$.

\vskip 5pt\noindent {\tt Case 2:} $H \cong \es$.
We have
\begin{align}
\vset(\add{G}{H}{\wh{H}})
&= \vset(\add{G}{\es}{\wh{H}})\\
&\eqlab{A}= \{\mex \vset(\add{G'}{\wh{H}}{\wh{H}}) \mid G' \in G\}\\
&\eqlab{B}= \{\mex ( (\vset(G')+ \wh{h}) \osum \vset(\wh{H}) ) \mid G' \in G\}\\
&\eqlab{C}= \{ (\mex \vset(G')) + \wh{h} \mid G' \in G\}\\
&= \{ \mex \vset(G')\mid G' \in G\} + \wh{h}\\
&\eqlab{D}= \vset(G) + \wh{h}\\
&= (\vset(G) + \wh{h}) \osum \es\\
&= (\vset(G) + \wh{h}) \osum \vset(\es), 
\end{align}
where
(A) follows from Equation \eqref{eq:iosum} and Proposition \ref{prop:vset} (ii),
(B) follows from the induction hypothesis by Equation \eqref{eq:44cf5enayvhl},
(C) follows from Lemma \ref{lem:mex-add},
and (D) follows from Proposition \ref{prop:vset} (ii).
\end{proof}

To see the outcome $o(\add{G}{H}{\wh{H}})$,
it suffices to see the outcome of each component as the following corollary to Theorem \ref{thm:iosum-vset} (cf.~Proposition \ref{prop:osum-outcome}).

\begin{corollary}
\label{cor:iosum-outcome}
For any $G, H, \wh{H} \in \short$, we have
\begin{align}
o(\add{G}{H}{\wh{H}}) =
\begin{cases}
\PP  &\,\,\text{if}\,\, o(H) = \PP \,\,\text{and}\,\, (o(G) = \PP \,\,\text{or}\,\, o(\wh{H}) = \NN),\\
\NN  &\,\,\text{otherwise}.
\end{cases}
\end{align}
\end{corollary}

\begin{proof}[Proof of Corollary \ref{cor:iosum-outcome}]
We consider the following two cases.

\vskip 5pt\noindent {\tt Case 1:} $o(H) = \NN$.
We have
\begin{align}
\vset(\add{G}{H}{\wh{H}})
\eqlab{A}= (\vset(G) + \mex\vset(\wh{H})) \osum \vset(H)
\eqlab{B}\supseteq (\vset(G) + \mex\vset(\wh{H})) \osum \{0\}
\eqlab{C}\owns 0,
\end{align}
where
(A) follows from Theorem \ref{thm:iosum-vset},
(B) follows from $\vset(H) \owns 0$ since $o(H) = \NN$,
and (C) follows from Definition \ref{def:osum-sets}.
This is equivalent to $o(\add{G}{H}{\wh{H}}) = \NN$.

\vskip 5pt\noindent {\tt Case 2:} $o(H) = \PP$.
We have
\begin{align}
o(\add{G}{H}{\wh{H}}) = \NN
&\iff \vset(\add{G}{H}{\wh{H}}) \owns 0\\
&\eqlab{A}\iff (\vset(G) + \mex\vset(\wh{H})) \osum \vset(H) \owns 0\\
&\eqlab{B}\iff (\vset(G) + \mex\vset(\wh{H})) \owns 0\\
&\iff \vset(G) \owns 0 \,\,\text{and}\,\, \mex\vset(\wh{H}) = 0\\
&\iff o(G) = \NN \,\,\text{and}\,\, o(\wh{H}) = \PP,
\end{align}
where
(A) follows from Theorem \ref{thm:iosum-vset},
and (B) follows from $\vset(H) \not\owns 0$ since $o(H) = \PP$.
\end{proof}

For $H$ and $\wh{H}$, only their Grundy numbers affect the Grundy number of $\add{G}{H}{\wh{H}}$ as follows.
This is a counterpart of Proposition \ref{prop:colon-pri}.

\begin{corollary}~
\label{cor:iosum-grundy}
\begin{enumerate}[label={\rm(\roman*)}]
\item For any $G, H_1, H_2, \wh{H} \in \short$,
if $\grundy(H_1) = \grundy(H_2)$, then $\grundy(\add{G}{H_1}{\wh{H}}) = \grundy(\add{G}{H_2}{\wh{H}})$.
\item For any $G, H, \wh{H}_1, \wh{H}_2 \in \short$,
if $\grundy(\wh{H}_1) = \grundy(\wh{H}_2)$, then $\grundy(\add{G}{H}{\wh{H}_1}) = \grundy(\add{G}{H}{\wh{H}_2})$.
\end{enumerate}
\end{corollary}

\begin{proof}[Proof of Corollary \ref{cor:iosum-grundy}]
(Proof of (i))
Assume $\grundy(H_1) = \grundy(H_2)$.
We have
\begin{align}
\grundy(\add{G}{H_1}{\wh{H}})
&\eqlab{A}= \mex\vset(\add{G}{H_1}{\wh{H}})\\
&\eqlab{B}= \mex\vset( (\add{G}{\es}{\wh{H}}) \osum H_1 )\\
&\eqlab{C}= \mex \left( \vset(\add{G}{\es}{\wh{H}}) \osum \vset(H_1) \right), \label{eq:umnk84vmwsed}
\end{align}
where
(A) follows from Proposition \ref{prop:vset} (i),
(B) follows from Proposition \ref{prop:iosum-osum} (iii),
and (C) follows from Theorem \ref{prop:vset-sum} (ii).
Similarly, we have
\begin{align}
\grundy(\add{G}{H_2}{\wh{H}}) = \mex \left( \vset(\add{G}{\es}{\wh{H}}) \osum \vset(H_2) \right). \label{eq:a3kmyjqgh0uk}
\end{align}
By the assumption $\grundy(H_1) = \grundy(H_2)$ and Lemma \ref{lem:set-osum},
the right-hand sides of Equations \eqref{eq:umnk84vmwsed} and \eqref{eq:a3kmyjqgh0uk} are identical as desired.

(Proof of (ii)) This is directly from Theorem \ref{thm:iosum-vset}.
\end{proof}

\subsection{Examples of Application}
\label{subsec:iosum-exam}

In this subsection, we provide several examples of application of the results above.

We first consider the case of a chain of nimbers $ \st a_0 \osum_{\st \wh{a}_1} \st a_1 \osum_{\st \wh{a}_2} \st a_2 \osum_{\st \wh{a}_3} \cdots \osum_{\st \wh{a}_n} \st a_n$.
We prove that its Grundy number is given as follows.

\begin{theorem}
\label{thm:iosum-star}
Let $n \in \pint$ and $a_0, a_1, a_2, \ldots, a_n, \wh{a}_1, \wh{a}_2, \ldots, \wh{a}_n \in \uint$ be arbitrary.
Let $p$ be the maximum $p' \in \{1, 2, \ldots, n\}$ such that $\sum_{i = p'}^n (a_i - \wh{a}_i) < 0$, where $p \defeq = 0$ if such $p'$ does not exist.
Then we have
\begin{align}
\grundy \left( \st a_0 \osum_{\st \wh{a}_1} \st a_1 \osum_{\st \wh{a}_2} \st a_2 \osum_{\st \wh{a}_3} \cdots \osum_{\st \wh{a}_n} \st a_n \right) = \sum_{i = p}^n a_i.
\end{align}
\end{theorem}

\begin{example}
Let $n = 6$, $(a_0, a_1, \ldots, a_n) = (6, 3, 9, 2, 1, 2, 5)$,
and $(\wh{a}_1, \wh{a}_2, \allowbreak\ldots, \wh{a}_n) \allowbreak = (10, 1, 5, 5, 0, 3)$.
Then $\sigma_{p'} \defeq= \sum_{i = p'}^n (a_i - \wh{a}_i)$ is given as
$(\sigma_1, \sigma_2, \ldots, \sigma_n)  \allowbreak= (4, -2, 5, -3, 0, 4, 2)$.
Thus, $p = 3$ and 
\begin{align}
\lefteqn{\grundy \left( \st 6 \osum_{\st 10} \st 3 \osum_{\st 1} \st 9 \osum_{\st 5} \st 2 \osum_{\st 5} \st 1 \osum_{\st 0} \st 2 \osum_{\st 3} \st 5\right)}\\
&= \grundy \left( \st a_0 \osum_{\st \wh{a}_1} \st a_1 \osum_{\st \wh{a}_2} \st a_2 \osum_{\st \wh{a}_3} \st a_3 \osum_{\st \wh{a}_4} \st a_4 \osum_{\st \wh{a}_5} \st a_5 \osum_{\st \wh{a}_6} \st a_6\right)\\
&\eqlab{A}= a_3+ a_4 + a_5 + a_6\\
&= 2 + 1 + 2 + 5\\
&= 10,
\end{align}
where (A) follows from Theorem \ref{thm:iosum-star}.
\end{example}

To prove Theorem \ref{thm:iosum-star},
we first prove the assertion for the particular case $n = 1$ as the following lemma.

\begin{lemma}
\label{lem:iosum-star}
For any $a, b, \wh{b} \in \uint$, we have
\begin{align}
\grundy(\st a \osum_{\st \wh{b}} \st b)
= \begin{cases}
a + b  &\,\,\text{if}\,\, b \geq \wh{b},\\
b  &\,\,\text{if}\,\, b < \wh{b}.
\end{cases}
\end{align}
\end{lemma}

\begin{proof}[Proof of Lemma \ref{lem:iosum-star}]
We have
\begin{align}
\vset(\st a \osum_{\st \wh{b}} \st b)
&\eqlab{A}= (\vset(\st a) + \grundy(\st \wh{b})) \osum \vset(\st b)\\
&= (\{0, 1, 2, \ldots, a-1\} + \wh{b}) \osum \{0, 1, 2, \ldots, b-1\}\\
&= \{\wh{b}, \wh{b} + 1, \wh{b} + 2, \ldots, \wh{b} +a-1\} \osum \{0, 1, 2, \ldots, b-1\}\\
&= \begin{cases}
\{0, 1, 2, \ldots, a+b-1\}  &\,\,\text{if}\,\, b \geq \wh{b},\\
\{0, 1, 2, \ldots, b-1, \wh{b}, \wh{b} + 1, \wh{b} + 2, \ldots, \wh{b} +a-1 \}  &\,\,\text{if}\,\, b < \wh{b},
\end{cases}
\end{align}
where
(A) follows from Theorem \ref{thm:iosum-vset}.
Since $\grundy(\st a \osum_{\st \wh{b}} \st b) = \mex\vset(\st a \osum_{\st \wh{b}} \st b)$, this shows the desired result.
\end{proof}

\begin{proof}[Proof of Theorem \ref{thm:iosum-star}]
We prove this by induction on $n$.
The base case $n = 1$ is shown as Lemma \ref{lem:iosum-star}.
For the induction step $n \geq 2$, we have
\begin{align}
\lefteqn{\grundy \left( \st a_0 \osum_{\st\wh{a}_1} \st a_1 \osum_{\st\wh{a}_2} \st a_2 \osum_{\st\wh{a}_3} \cdots \osum_{\st\wh{a}_n} \st a_n \right)}\\
&\eqlab{A}= \grundy \left( \st a_0 \osum_{(\st\wh{a}_1 \osum_{\st\wh{a}_2} \st\wh{a}_2 \osum_{\st\wh{a}_3} \cdots \osum_{\st\wh{a}_n} \st\wh{a}_n)} (\st a_1 \osum_{\st \wh{a}_2} \st a_2 \osum_{\st \wh{a}_3} \cdots \osum_{\st \wh{a}_n} \st a_n) \right)\\
&\eqlab{B}= \grundy \left( \st a_0 \osum_{\st(\sum_{i = 1}^n \wh{a}_i)}  (\st a_1 \osum_{\st \wh{a}_2} \st a_2 \osum_{\st\wh{a}_3} \cdots \osum_{\st \wh{a}_n} \st a_n) \right)\\
&\eqlab{C}= \grundy \left( \st a_0 \osum_{\st(\sum_{i = 1}^n \wh{a}_i)} \st\left(\textstyle\sum_{i = \max\{1, p\}}^n a_i\right) \right)\\
&\eqlab{D}= \begin{cases}
a_0 + \textstyle\sum_{i = \max\{1, p\}}^n a_i  &\,\,\text{if}\,\, \textstyle\sum_{i = \max\{1, p\}}^n a_i \geq \textstyle\sum_{i = 1}^n \wh{a}_i,\\
\textstyle\sum_{i = \max\{1, p\}}^n a_i  &\,\,\text{if}\,\, \textstyle\sum_{i = \max\{1, p\}}^n a_i < \textstyle\sum_{i = 1}^n \wh{a}_i
\end{cases}\\
&\eqlab{E}= \begin{cases}
a_0 + \textstyle\sum_{i = \max\{1, p\}}^n a_i  &\,\,\text{if}\,\, p = 0,\\
\textstyle\sum_{i = \max\{1, p\}}^n a_i  &\,\,\text{if}\,\, p \geq 1
\end{cases}\\
&= \begin{cases}\textstyle\sum_{i = 0}^n a_i  &\,\,\text{if}\,\, p = 0,\\
\textstyle\sum_{i = p}^n a_i  &\,\,\text{if}\,\, p \geq 1
\end{cases}\\
&= \textstyle\sum_{i = p}^n a_i
\end{align}
as desired, where
(A) follows from Corollary \ref{cor:iosum},
(B) follows from the induction hypothesis and Corollary \ref{cor:iosum-grundy} (ii),
(C) follows from the induction hypothesis and Corollary \ref{cor:iosum-grundy} (i),
(D) follows from Lemma \ref{lem:iosum-star},
and (E) follows because $\sum_{i = \max\{1, p\}}^n a_i \geq \textstyle\sum_{i = 1}^n \wh{a}_i$ if and only $p = 0$ as seen as follows.

If $p = 0$, then by the definition of $p$, for all $p' = 1, 2, \ldots, n$,
it holds that $\sum_{i = p'}^n a_i \geq \textstyle\sum_{i = p'}^n \wh{a}_i$;
in particular, $\sum_{i = \max\{1, p\}}^n = \sum_{i = 1}^n a_i \geq \textstyle\sum_{i = 1}^n \wh{a}_i$ as desired.
Conversely, if $p \geq 1$, then
\begin{align}
\sum_{i = \max\{1, p\}}^n a_i
= \sum_{i = p}^n a_i
\eqlab{A}< \sum_{i = p}^n \wh{a}_i
\leq \sum_{i = 1}^n \wh{a}_i,
\end{align}
where (A) follows from the definition of $p$.
\end{proof}

As another example, we next consider the following generalization of ordinal sums with substitution to a poset structure.
\begin{quote}
A finite poset $(P, \preceq)$ is given (for example, as shown in Figure \ref{fig:poset}).
To each element $x \in P$, an impartial game $G_x$ is assigned independently. 
On a turn, a player chooses an arbitrary $x \in P$ and makes a move on the associated game $G_x$.
At that time, for every $y \in P$ such that $y \preceq x$ and $y \neq x$,
the game $G_y$ is replaced with a game $\wh{G}_y$, which is fixed in advance for each $y \in P$.
\end{quote}

\begin{figure}
\centering
\includegraphics[keepaspectratio,scale=0.3]{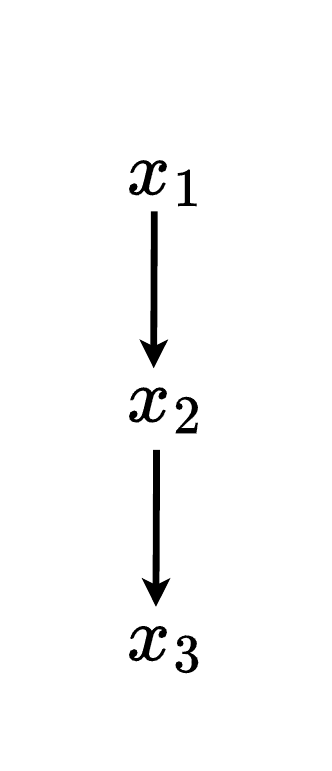}
\includegraphics[keepaspectratio,scale=0.3]{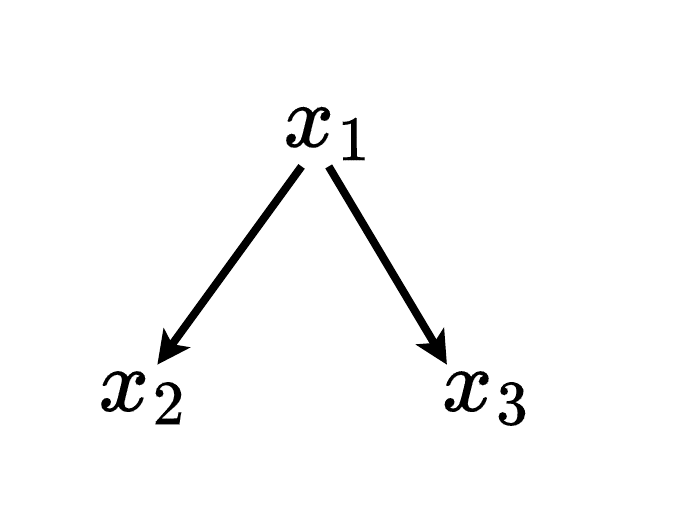}
\includegraphics[keepaspectratio,scale=0.3]{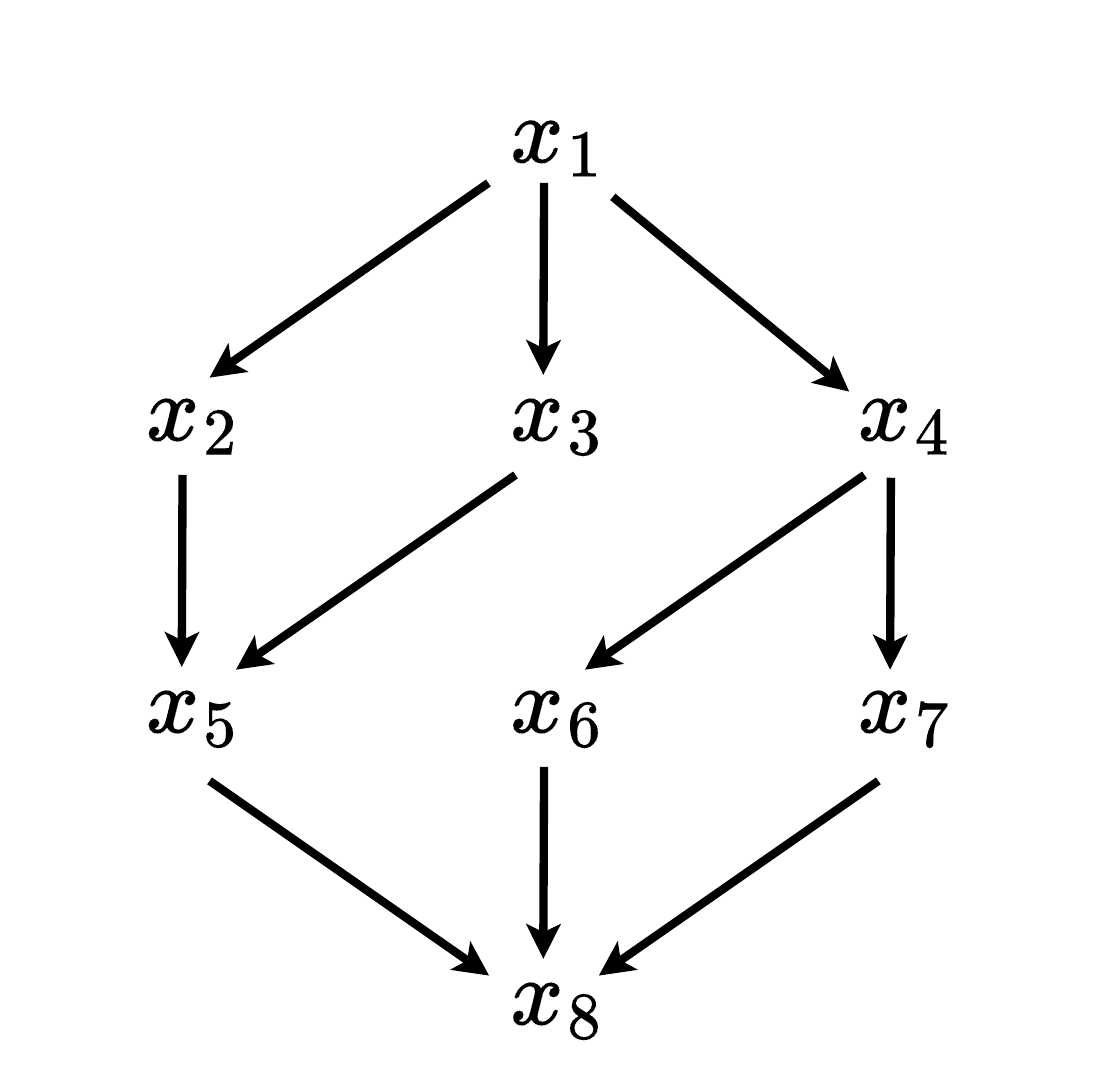}
\caption{The Hasse diagrams of three posets.}
\label{fig:poset}
\end{figure}

For example, when the given poset is the right one in Figure \ref{fig:poset}, and
a player makes a move on $G_4$ to an option $G'_4$, then
the three games $G_6, G_7, G_8$ associated to $x_6, x_7, x_8$ are replaced with $\wh{G}_6, \wh{G}_7, \wh{G}_8$,
respectively, since the smaller elements than $x_4$ are $x_6, x_7, x_8$.

Depending on the structure of a given poset,
the corresponding game may be represented by recursively taking disjunctive sums and ordinal sums,
as the following example.

\begin{example}
\label{ex:poset-exp}
If the left poset in Figure \ref{fig:poset} is given, then
the initial position is represented as $G_1 \osum_{\wh{G}_2} G_2 \osum_{\wh{G}_3} G_3$.
For the poset in the center of Figure \ref{fig:poset},
the initial position is represented as $G_1 \osum_{(\wh{G}_2 + \wh{G}_3)} (G_2 + G_3)$.
For a more complicated example, the initial position corresponding to the right one in Figure \ref{fig:poset} is
represented as 
$G_1 \osum_{\wh{H}} H \osum_{\wh{G}_8} G_8$,
where
\begin{align}
H &\defeq\cong ((G_2 + G_3) \osum_{\wh{G}_5} G_5) + (G_4 \osum_{(\wh{G}_6 + \wh{G}_7)} (G_6 + G_7)),\\
\wh{H} &\defeq\cong ((\wh{G}_2 + \wh{G}_3) \osum_{\wh{G}_5} \wh{G}_5) + (\wh{G}_4 \osum_{(\wh{G}_6 + \wh{G}_7)} (\wh{G}_6 + \wh{G}_7)).
\end{align}
\end{example}

If the structure of the poset can be expressed in terms of disjunctive sums and ordinal sums,
we can use Theorem \ref{thm:iosum-vset} to determine the variation set of the game as in the following example.

\begin{figure}
\centering
\includegraphics[keepaspectratio,scale=0.33]{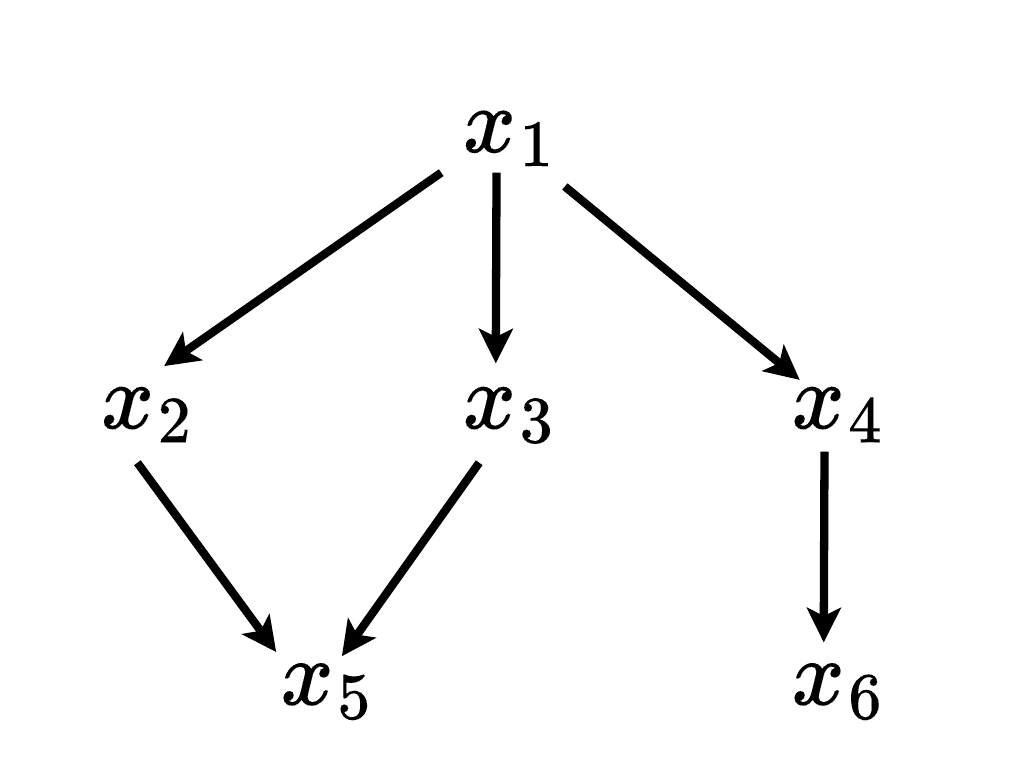}
\caption{The Hasse diagram of a poset.}
\label{fig:poset-exam}
\end{figure}

\begin{example}
Suppose that the poset in Figure \ref{fig:poset-exam} is given.
Also, let $G_i \defeq\cong \st i$ and $\wh{G}_i \defeq\cong \st (7-i)$ for $i \in \{1, 2, 3, 4, 5, 6\}$.
Then the corresponding game is represented as
\begin{align}
\lefteqn{G_1 \osum_{ \left( (\wh{G}_2 + \wh{G}_3) \osum_{\wh{G}_5} \wh{G}_5 \right) + \left(\wh{G}_4 \osum_{\wh{G}_6} \wh{G}_6 \right) }
\left( \left( (G_2 + G_3) \osum_{\wh{G}_5} G_5 \right) + \left(G_4 \osum_{\wh{G}_6} G_6 \right)  \right)}\\
&\cong \st 1 \osum_{ \left( (\st 5 + \st 4) \osum_{\st 2} \st 2 \right) + \left(\st 3 \osum_{\st 1} \st 1 \right) }
\left( \left( (\st 2 + \st 3) \osum_{\st 2} \st 5 \right) + \left(\st 4 \osum_{\st 1} \st 6 \right) \right),
\end{align}
equivalently, $\add{\st 1}{H}{\widehat{H}}$, where
\begin{align}
H &\defeq\cong \left((\st 2 + \st 3) \osum_{\st 2} \st 5 \right) + \left(\st 4 \osum_{\st 1} \st 6\right),\\
\widehat{H} &\defeq\cong \left((\st 5 + \st 4) \osum_{\st 2} \st 2 \right) + \left(\st 3 \osum_{\st 1} \st 1\right).
\end{align}
We have
\begin{align}
\vset(\st 2 + \st 3)
&\eqlab{A}= \vset(\st 2) \oplus \grundy(\st 3) \cup  \vset(\st 3) \oplus \grundy(\st 2)\\
&= \{0, 1\} \oplus 3 \cup  \{0, 1, 2\} \oplus 2\\
&= \{3, 2\} \cup \{2, 3, 0\}\\
&= \{0, 2, 3\}, \label{eq:sfbou70oiwry}
\end{align}
where
(A) follows from Proposition \ref{prop:vset-sum} (i).
Hence, we have
\begin{align}
\vset\left( (\st 2 + \st 3) \osum_{\st 2} \st 5 \right)
&\eqlab{A}= (\vset(\st 2 + \st 3) +  \grundy(\st 2)) \osum \vset(\st 5)\\
&\eqlab{B}= (\{0, 2, 3\} + 2) \osum \{0, 1, 2, 3, 4\}\\
&= \{2, 4, 5\} \osum \{0, 1, 2, 3 ,4\}\\
&= \{0, 1, 2, 3, 4, 5, 6, 7\}, \label{eq:rqo5aj5tvp8u}
\end{align}
where
(A) follows from Theorem \ref{thm:iosum-vset},
and (B) follows from Equation \eqref{eq:sfbou70oiwry}.
Also, we have
\begin{align}
\vset\left( \st4 \osum_{\st 1} \st 6 \right)
&\eqlab{A}= (\vset(\st 4) +  \grundy(\st 1)) \osum \vset(\st 6)\\
&= (\{0, 1, 2, 3\} + 1) \osum \{0, 1, 2, 3, 4, 5\}\\
&= \{1, 2, 3, 4\} \osum \{0, 1, 2, 3 ,4, 5\}\\
&= \{0, 1, 2, 3, 4, 5, 6, 7, 8, 9\}, \label{eq:m0ihic0wj5xe}
\end{align}
where 
(A) follows from Theorem \ref{thm:iosum-vset}.
Therefore, we obtain
\begin{align}
\lefteqn{\vset(H)}\\
&= \vset\left( ((\st 2 + \st 3) \osum_{\st 2} \st 5) + (\st 4 \osum_{\st 1} \st 6) \right)\\
&\eqlab{A}= \vset\left( (\st 2 + \st 3) \osum_{\st 2} \st 5\right) \oplus \grundy\left(\st 4 \osum_{\st 1} \st 6 \right)
\cup \vset\left(\st 4 \osum_{\st 1} \st 6 \right) \oplus \grundy\left( (\st 2 + \st 3) \osum_{\st 2} \st 5\right)\\
&\eqlab{B}= \{0, 1, 2, 3, 4, 5, 6, 7\} \oplus 10 \cup \{0, 1, 2, 3, 4, 5, 6, 7, 8, 9\} \oplus 8\\
&= \{10, 11, 8, 9, 14, 15, 12, 13\} \cup \{8, 9, 10, 11, 12, 13, 14, 15, 0, 1\}\\
&= \{0, 1, 8, 9, 10, 11, 12, 13, 14, 15\}, \label{eq:pib1fq4b18dt}
\end{align}
where
(A) follows from Proposition \ref{prop:vset-sum} (i),
and (B) follows from Equations \eqref{eq:rqo5aj5tvp8u} and \eqref{eq:m0ihic0wj5xe}.
On the other hand, we have
\begin{align}
\vset(\st 5 + \st 4)
&\eqlab{A}= \vset(\st 5) \oplus \grundy(\st 4) \cup  \vset(\st 4) \oplus \grundy(\st 5)\\
&= \{0, 1, 2, 3, 4\} \oplus 4 \cup  \{0, 1, 2, 3\} \oplus 5\\
&= \{4, 5, 6, 7, 0\} \cup \{5, 4, 7, 6\}\\
&= \{0, 4, 5, 6, 7\}, \label{eq:1hy0zxlijp6j}
\end{align}
where
(A) follows from Proposition \ref{prop:vset-sum} (i).
Hence, 
\begin{align}
\vset\left( (\st 5 + \st 4) \osum_{\st 2} \st 2 \right)
&\eqlab{A}= (\vset(\st 5 + \st 4) +  \grundy(\st 2)) \osum \vset(\st 2)\\
&\eqlab{B}= (\{0, 4, 5, 6, 7\} + 2) \osum \{0, 1\}\\
&= \{2, 6, 7, 8, 9\} \osum \{0, 1\}\\
&= \{0, 1, 2, 6, 7, 8, 9\}, \label{eq:mfowjxaz37at}
\end{align}
where
(A) follows from Theorem \ref{thm:iosum-vset},
and (B) follows from Equation \eqref{eq:1hy0zxlijp6j}.
Therefore, we obtain
\begin{align}
\grundy(\widehat{H})
&= \grundy\left( ((\st 5 + \st 4) \osum_{\st 2} \st 2) + (\st 3 \osum_{\st 1} \st 1) \right)\\
&\eqlab{A}= \grundy\left( (\st 5 + \st 4) \osum_{\st 2} \st 2\right) \oplus \grundy\left(\st 3 \osum_{\st 1} \st 1 \right)\\
&\eqlab{B}= 3 \oplus 4\\
&= 7, \label{eq:cucuwedv6to5}
\end{align}
where
(A) follows from Proposition \ref{prop:grundy} (ii),
and (B) follows from Equation \eqref{eq:mfowjxaz37at} and Theorem \ref{thm:iosum-star}.
Combining the results above, we finally obtain
\begin{align}
\vset\left(\add{\st 1}{H}{\widehat{H}} \right)
&\eqlab{A}= (\vset(\st 1) +  \grundy(\widehat{H})) \osum \vset(H)\\
&\eqlab{B}= (\{0\} + 7) \osum \{0, 1, 8, 9, 10, 11, 12, 13, 14, 15\}\\
&= \{7\} \osum \{0, 1, 8, 9, 10, 11, 12, 13, 14, 15\}\\
&= \{0, 1, 7, 9, 10, 11, 12, 13, 14, 15, 16\},
\end{align}
where
(A) follows from Theorem \ref{thm:iosum-vset},
and (B) follows from Equations \eqref{eq:pib1fq4b18dt} and \eqref{eq:cucuwedv6to5}.
In particular, we have $\grundy(\add{\st 1}{H}{\widehat{H}}) = 2$ and $o\left(\add{\st 1}{H}{\widehat{H}} \right) = \NN$.
\end{example}

\end{document}